\documentclass[a4paper,12pt]{article}
\usepackage[margin=1.2in]{geometry} 
\usepackage{amsfonts} 
\usepackage{amsmath} 
\usepackage{amssymb} 
\usepackage{amsthm} 
\usepackage{algorithm} 
\usepackage{algorithmic}
\usepackage{lipsum}
\usepackage{color}
\usepackage{hyperref}
\usepackage{graphicx}

\newtheorem{theorem} {Theorem}
\newtheorem{lemma} {Lemma}
\newtheorem{definition} {Definition}

\newtheorem{assumption} {Assumption}

\def\e{{\mathbf{e}}}

\DeclareMathOperator*{\argmin}{arg\,min}

\newcommand{\mY}{\mathcal{Y}}
\newcommand{\mX}{\mathcal{X}}
\newcommand{\tY}{\tilde{\mathcal{Y}}}
\newcommand{\tX}{\tilde{\mathcal{X}}}
\newcommand{\tS}{\tilde{S}}

\newcommand{\mA}{\mathcal{A}}
\newcommand{\mB}{\mathcal{B}}

\newcommand{\mK}{\mathcal{K}}

\newcommand{\reals}{\mathbb{R}}

\newcommand{\oraclesp}{\mathcal{O}_{SP}}
\newcommand{\oraclek}{\mathcal{O}_{\mK}}
\newcommand{\oraclex}{\mathcal{O}_{\mX}}
\newcommand{\oracley}{\mathcal{O}_{\mY}}

\begin{document}
\title{Blackwell's Approachability with Approximation Algorithms}
\author{Dan Garber\footnote{dangar@technion.ac.il} 
\and
Mhna Massalha\footnote{mhnamassalha@campus.technion.ac.il} \\ 
}
\date{Faculty of Data and Decision Sciences  \vspace{5pt} \\ Technion - Israel Institute of Technology}
\maketitle

\begin{abstract}
We revisit Blackwell's celebrated approachability problem which considers a repeated vector-valued game between a player and an adversary. Motivated by settings in which the action set of the player or adversary (or both) is difficult to optimize over, for instance when it corresponds to the set of all possible solutions to some NP-Hard optimization problem, we ask what can the player guarantee \textit{efficiently}, when only having access to these sets via approximation algorithms with ratios $\alpha_{\mX} \geq 1$ and $ 1 \geq \alpha_{\mY} > 0$, respectively. Assuming the player has monotone preferences, in the sense that he does not prefer a vector-valued loss $\ell_1$ over $\ell_2$ if $\ell_2 \leq \ell_1$, we establish that given a Blackwell instance with an approachable target set $S$,  the downward closure of the appropriately-scaled set $\alpha_{\mX}\alpha_{\mY}^{-1}S$ is \textit{efficiently} approachable with optimal rate. In case only the player's or adversary's set is equipped with an approximation algorithm, we give simpler and more efficient algorithms.
\end{abstract}

\section{Introduction}
In his seminal paper from the 1950s, Blackwell presented his celebrated approachability problem and theorem \cite{blackwell1956analog}. Blackwell considered repeated vector-valued games in which on each iteration $t$, the first player, which we shall simply refer to as \textit{the player}, picks an action $x_t$ from a fixed compact and convex set $\mX$, which is then followed by the second player, which we shall refer to as \textit{the adversary}, choosing his action $y_t$ from another fixed compact and convex set $\mY$, and then, a corresponding vector-valued loss $\ell(x_t,y_t)$ is revealed, where we shall assume throughout that $\ell$ is a fixed and known bi-linear function. The game then continues to the next round. Focusing on the view-point of the player, his goal in this game is to obtain average (over time) loss, whose distance (say in Euclidean norm) to some fixed and pre-specified convex and closed target set $S$, goes to zero as the length of the game goes to infinity, no matter how the adversary plays. In case this goal is feasible, the target set $S$ is said to be \textit{approachable}. Blackwell gave both a characterization when a set $S$ is approachable, and an algorithm for choosing the actions of the player so that approachability is indeed obtained. 

The literature on approachability is vast and we do not presume to appropriately survey it here. Instead we refer the interested reader to \cite{maschler2020game, perchet2014approachability, cesa2006prediction} and references therein.

Algorithms for choosing the player's actions so that approachability to the target set is guaranteed (when possible), naturally require, as a sub-procedure, to solve certain optimization problems with respect to the actions sets of the player and the adversary (and also the target set $S$). In this work however, we are interested in the case in which such optimizations may be computationally-infeasible. For instance, when one of these sets, $\mX$ or $\mY$ (or both), corresponds to the set of all possible solutions to some NP-Hard optimization problem (or more accurately, the convex hull of all such solutions, to include mixed-actions). As an example, we can think that the action the player needs to take on each round of the game is to select a vertex cover or a Hamiltonian cycle in some fixed combinatorial graph, or that the adversary's action corresponds to choosing a cut in some fixed combinatorial graph (see a more detailed example in the sequel). In such a case, while the Blackwell instance might be approachable, it might not be efficiently approachable, since no polynomial-time algorithm that can guarantee approachability is known (or even widely expected to exist). Instead, in such a case it is natural to assume that the action sets (those for which optimization is difficult) are accessible through a computationally-efficient \textit{approximation algorithm} (approximation oracle henceforth) with some ratio $\alpha > 0$. If $\alpha > 1$, then for any non-negative linear objective, the oracle returns a point in the action set whose value is at most $\alpha$ times that of the action with minimal value; if $\alpha < 1$,  the oracle returns a point in the action set whose value is at least $\alpha$ times that of the action with maximal value.
The question is then, what can the player guarantee (w.r.t. his average loss) using only a polynomial (in the natural parameters of the problem) number of calls  to these oracles  (and additional polynomial running time). 

As one illustrative example, we can consider the following game in which there is a fixed undirected graph $G(V,E)$, where the set of vertices $V$ is partitioned into two subsets $V_1, V_2$, each is controlled by a different interested party. The player, representing both parties, is required on each iteration to choose action $x\in\mX\subset[0,1]^{|V|}$ which is a distribution (to allow mixed actions) over vertex covers of the graph\footnote{here a vertex cover is represented by its identfiyng vector which is an element in $\{0,1\}^{|V|}$}. The adversary in turn chooses $y\in\mY\subset\reals^{|V|}_+$ which assigns a (bounded) non-negative weight to each vertex. The bilinear loss is then a vector in $\reals^2_+$ which assigns to each party $i\in\{1,2\}$ its weighted loss, i.e., $\ell(x,y) = (\sum_{j\in{}V_1}x_jy_j, \sum_{j\in{}V_2}x_jy_j)^{\top}$. The 
pre-defined target set $S$ is designed to balance between the interests of both parties. While minimum-weight vertex cover is NP-Hard, there are well known efficient $2$-approximation algorithms, e.g., \cite{bar1981linear}.

As we shall discuss in the sequel, and as one might expect, even when given a Blackwell instance with an approachable target set $S$, in case we only have access to $\mX$ or $\mY$ via an approximation oracle, approachability to $S$ can no longer be guaranteed. One then might expect that perhaps there is an appropriate scaling phenomena: for instance, if the action set of the player is given by an approximation oracle with ratio $\alpha_{\mX} \geq 1$ and the target set $S$ is non-negative and approachable, then the set $\alpha_{\mX}{}S$ is efficiently approachable. However,  this is also not the case in general. Instead, we shall introduce a novel approachability goal suitable for our computationally-constrained setting. 

We shall assume that the player's preferences are monotone, in the sense that given two vector-valued outcomes $\ell_1,\ell_2$, $\ell_1$ will not be preferable to $\ell_2$ if $\ell_2 \leq \ell_1$ (coordinate-wise). As our most general result, we prove that if the original target set is approachable then,  when the player's set is accessible through an approximation oracle with ratio $\alpha_{\mX} > 1$, and the adversary's set is accessible though an approximation oracle with ratio $1 > \alpha_{\mY} > 0$, there is an efficient algorithm for the player such that the \textit{downward closure of the scaled target set $\alpha_{\mX}\alpha_{\mY}^{-1}S$} is approachable.  In other words: as the length of the game $T$ goes to infinity, the average loss of the player becomes arbitrarily close to some vector that is at least as good (in the sense $\leq$) as $\alpha_{\mX}\alpha_{\mY}^{-1}$ times some vector in the original set  $S$. Moreover, the approachability rate is optimal, i.e., $O(1/\sqrt{T})$. In case only the player's action set or the adversary's action set are accessible through an approximation oracle (without restriction on the other set), we give significantly simpler and more efficient algorithms.

Our approach builds heavily on several techniques from  \textit{online convex optimization} (OCO) \cite{DBLP:journals/corr/abs-1909-05207}. First, we use the approach introduced in \cite{abernethy2011blackwell} and later refined in \cite{shimkin2016online}, to reduce approachability to online convex optimization over the Euclidean unit ball. This approach however requires access to what is referred to as an \textit{halfspace oracle} \cite{abernethy2011blackwell}, which in principle could be implemented by solving a saddle-point optimization problem w.r.t. the action sets $\mX,\mY$, however cannot be efficiently implemented in our setting. Instead, we introduce what we refer to as an \textit{infeasible saddle-point oracle}, which in a sense, solves these problems w.r.t. a modified Blackwell instance, while providing guarantees that are feasible w.r.t. the original instance. Implementing these new oracles also uses machinery from  OCO to implement primal-dual methods (which is a well-studied technique, see for instance \cite{hazan2006approximate})  in a novel way that can solve these infeasible saddle-point problems via access to approximation oracles.

Finally, our work is  related, and in fact very much inspired by the problem of \textit{online linear optimization with approximation algorithms}, studied   in \cite{kakade2007playing} and later revisited in \cite{garber2017efficient} and \cite{hazan2018online}. In this line of work, it is shown how to turn an approximation oracle with ratio $\alpha$ for an offline linear optimization problem, into an efficient algorithm for the corresponding online linear optimization problem with approximately the same approximation guarantee, i.e.,  to guarantee sub-linear regret w.r.t. $\alpha$ times the cumulative loss/payoff of the best action in hindsight. While, to the best of our understanding, these algorithms cannot be applied in our approachability setting, some techniques introduced in \cite{kakade2007playing}, such as the extended approximation oracle and the use of the Frank-Wolfe technique for computing so-called infeasible projections, are also used as building blocks in the construction of our algorithms.

\section{Preliminaries}
\subsection{Notation}
We use $\Vert{\cdot}\Vert$ to denote the Euclidean norm for real column vectors, and the spectral norm (largest singular value) for real matrices. We let $\langle{\cdot,\cdot}\rangle$ denote the standard inner-product in a Euclidean space. 
We denote by $\mB(R)$ the origin-centered Euclidean ball of radius $R>0$ in some finite-dimensional Euclidean space (which will be clear from context), and we let $\mB_+(R)$ denote its restriction to the nonnegative orthant. In case $R=1$, we will simply write $\mB$ and $\mB_+$. 
In the following let $\mK\subseteq\reals^d$ be some closed and convex set.
If $\mK$ is bounded,  we shall denote by $R_{\mK}$ the radius of some enclosing origin-centered Euclidean ball, i.e., $R_{\mK}\geq \max_{x \in \mK} \|x\|$, and we shall denote by $D_{\mK}$ some upper-bound on its diameter, i.e., $D_{\mK}\geq \max_{x,y \in \mK} \|x-y\|$.
We shall denote the downward closure of $\mK$ as:
\begin{align*}
 (\mK)_{\downarrow} := \{y\in\reals^d ~|~\exists x\in\mK: y\leq x\}.
\end{align*}

We denote by $d(x, \mK)$ the Euclidean distance between a point $x \in \reals^d$ and the set $\mK$. The support function of $\mK$, $h_{\mK} : \reals^d \rightarrow \mathbb{R} \cup \{ \infty \}$, is given by
\begin{align}
\label{support function}
h_{\mK}(x):=\sup_{k \in \mK}\langle x , k \rangle, \quad x \in \reals^d.
\end{align} 
$d(x, \mK)$  can then be expressed as (see for instance \cite{boyd2004convex}, Section 8.1.3):
\begin{equation}\label{dif:dis}
d(x, \mK) = \max_{w \in \mB} \{ \langle w, x \rangle - h_{\mK}(w) \}.
\end{equation}

\subsection{Blackwell's Approachability}
We first briefly review some central concepts of approachability theory.
\begin{definition}[Blackwell instance] 
\label{blackwell instance}
A Blackwell instance is a tuple $(\mX,\mY,\ell,S)$ such that $\mX \subset \reals^n$ (action set of player), $\mY \subset \reals^m$ (action set of adversary) are convex and compact sets, the target set $S \subset \reals^d$ is convex and closed, and the vector-valued loss function $\ell :\reals ^n \times \reals ^m \to \reals ^d $ is bi-linear.
\end{definition}

\begin{definition}[Approachable set]
Let $(\mX,\mY,\ell,S)$ be a Blackwell instance. The set $S$ is said to be approachable, if there exists an algorithm $\mathcal{A}$ which selects points in $\mX$ such that, for any sequence $(y_t)_{t\geq 1}\subseteq\mY$, we have
\begin{align*}
d\left({\frac{1}{T}\sum_{t=1}^T\ell(x_t,y_t), S}\right)\rightarrow 0 \quad \textrm{as} \quad T\rightarrow \infty,
\end{align*}
where $x_{t} \gets \mA(y_1,\dots,y_{t-1})$.
\end{definition}
\begin{theorem}[Blackwell's approachability theorem]\label{thm:approach}
Let $(\mX ,\mY, \ell ,S)$ be a Blackwell instance. Then, it is an approachable instance (i.e., the  set $S$ is approachable) if and only if, for any $w \in \mB$ there exists an action $x \in \mX$ such that,
\begin{equation}\label{cond}
\max_{y\in\mY}\langle w, \ell(x, y) \rangle - h_S(w) \leq 0.
\end{equation}
\end{theorem}
When \eqref{cond} is feasible,  an action $x\in\mX$ corresponding to \eqref{cond} can be computed by solving the corresponding bi-linear saddle-point optimization problem: 
\begin{align}\label{eq:saddlepoint}
\min_{x\in\mX}\max_{y\in\mY}\{g(x,y):=\langle w, \ell(x, y) \rangle\}.
\end{align}
Given access to an oracle for computing points in $\mX$ satisfying Eq. \eqref{cond}, Blackwell gave a simple and elegant algorithm for choosing the actions of the player so that the average loss indeed approaches $S$ and with the optimal rate $O(1/\sqrt{T})$, see \cite{blackwell1956analog, cesa2006prediction}.
\section{Setup and Main Result}

\begin{definition}[approximation oracle] 
\label{def:app oracle}
Given a convex, compact, and non-negative set $\mK\subset\reals^n_{+}$, we let $\mathcal{O}_{\mK}$ denote an approximation oracle with ratio $\alpha_{\mK} > 0$, if for any non-negative vector $c\in\reals^n_+$, $\mathcal{O}_{\mK}(c)\in \mK$, and
\begin{align*}
&  \langle{\mathcal{O}_{\mK}(c), c}\rangle \leq \alpha_{\mK}\min_{x\in\mK} \langle{x,c}\rangle  ~~ \textrm {if} ~~ \alpha_{\mK} \geq 1;   \qquad \langle{\mathcal{O}_{\mK}(c), c}\rangle \geq \alpha_{\mK}\max_{x\in\mK} \langle{x,c}\rangle  ~~ \textrm {if} ~~ \alpha_{\mK} < 1. 
\end{align*} 
\end{definition} 
Henceforth we shall assume one of the following three scenarios: either $\mX$ or $\mY$ are given by an approximation oracle with ratio $\alpha_{\mX} >1$ or $1 > \alpha_{\mY} > 0$, respectively, and the other set is unrestricted (concretely we shall simply assume that Euclidean projections are computationally efficient over the other set), or both $\mX,\mY$ are given by approximation oracles. In all three scenarios we shall assume $S$ is equipped with a (exact) linear optimization oracle, which is required to compute the support function $h_S(w)$ (as defined in \eqref{support function}) and its subgradients.

We shall make the following assumption regarding the considered Blackwell instance.
\begin{assumption}\label{ass:1}
Given a Blackwell instance $(\mX,\mY,\ell,S)$ (Definition \ref{blackwell instance}), we assume  that $\mX\subset\reals^n_+$, $\mY\subset\reals^m_+$, $S$ is compact\footnote{The assumption that $S$ is bounded is in a sense without loss of generality, since under the assumption that $\mX,\mY$ are compact, so is the image of $\ell$ over $\mX\times\mY$ and thus we can  effectively replace $S$ with $S\cap{}S_0$, where $S_0 := \{\ell(x,y) ~|~(x,y)\in\mX\times\mY\}$.} and approachable set, and $\ell$ is non-negative over $\reals^n_+\times\reals^m_+$. 
\end{assumption}

Additionally, for the sake of simplifying the presentation, we shall also assume the length of the game $T$ is finite and known in advance. Removing this assumption is quite straightforward and does not pose a significant  technical challenge.


\paragraph*{What can be approached?}
Consider the Blackwell instance  $(\mX,\mY,\ell,S)$ given by $\mX = \left\{(z,1)^{\top}~|~z\in[1,\alpha_{\mX}]\right\}\subset\reals^2$ for some scalar $\alpha_{\mX} > 1$, $\mY=\{1\}$,  $\ell:\reals^2\times\reals\rightarrow\reals^2$ given by $\ell(x,y) = (x(1)y, x(2)y)^{\top}$, and $S = \{(1,1)\}$, which is clearly approachable. Suppose now that $\mX$ is accessible only through an approximation oracle $\oraclex(\cdot)$ such that for any $c \geq 0$, $\oraclex(c) = (\alpha_{\mX}, 1)^{\top}\in\mX$. Clearly, $\oraclex$ has approximation ratio $\alpha_{\mX}$. However, given access to $\mX$ only through this oracle, the player  is forced to always play $x_t=(\alpha_{\mX},1)^{\top}$ and clearly cannot approach $S$. Note  he also cannot approach $\alpha_{\mX}S = \{(\alpha_{\mX},\alpha_{\mX})\}$  (nor the convex-hull of their union, $\text{convex-hull}\{S\cup\alpha_{\mX}S\}$).

In order to bypass this simple negative result, we propose an alternative goal for the player. We consider the player's preferences to be monotone, in the sense that given two vector-valued outcomes $\ell_1,\ell_2$, $\ell_1$ is not preferable to $\ell_2$ if $\ell_2 \leq \ell_1$. Taking this view,  we propose approaching the \textit{downward closure of $\alpha_{\mX}S$}, that is $\left({\alpha_{\mX}S}\right)_{\downarrow}$. This means, the player is satisfied with being able to \textit{efficiently} approximately guarantee average loss that is not worse than $\alpha_{\mX}$ times some loss vector he had been satisfied with, had he had access to unlimited computation.

More generally, when both actions sets $\mX,\mY$ are accessible only through approximation oracles, we shall consider the task of approaching the set  $\left({\alpha_{\mX}\alpha_{\mY}^{-1}S}\right)_{\downarrow}$.
\begin{figure}[H]
\centering
\includegraphics[width=0.33\textwidth]{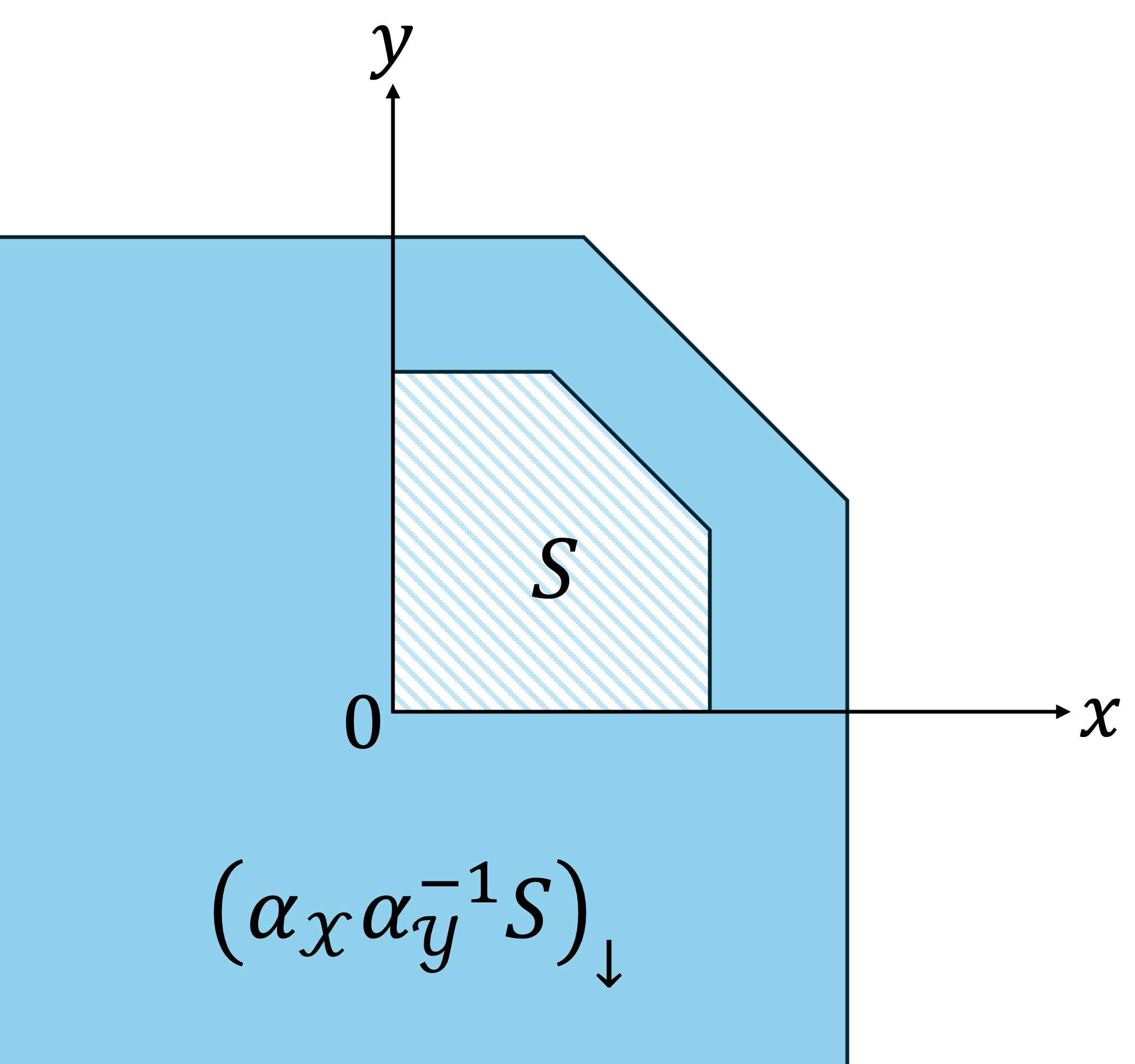}
\caption{Illustration of the relationship between the target set $S$ and the downward closure of its scaling by the joint approximation ratio $\alpha_{\mX}\alpha_{\mY}^{-1}$, $(\alpha_{\mX}\alpha_{\mY}^{-1}S)_{\downarrow}$.}
\end{figure}

We can now state our main results, at least at a high-level.
\begin{theorem}[main result]\label{thm:main}
Let $(\mX,\mY,\ell,S)$ be a Blackwell instance that satisfies Assumption \ref{ass:1}. 
\begin{enumerate}
\item
If the player's set $\mX$ (alternatively, the adversary's set $\mY$) is accessible through an approximation oracle with ratio  $\alpha_{\mX} \geq 1$ ($1\geq \alpha_{\mY} > 0$)  (and with no restriction on $\mY$ ($\mX$)), then $(\alpha_{\mX}S)_{\downarrow}$ is approachable ($(\alpha_{\mY}^{-1}S)_{\downarrow}$ is approachable). In particular, there is an algorithm with approachability rate $O(1/\sqrt{T})$ which requires a  total number of $O(T^2)$ calls to the approximation oracle.
\item
If both $\mX$ and $\mY$ are accessible through  approximation oracles with ratios $\alpha_{\mX} \geq 1$ and  $1 \geq \alpha_{\mY} >0$, respectively, then $(\alpha_{\mX}\alpha_{\mY}^{-1}S)_{\downarrow}$ is approachable. In particular, there is an algorithm with approachability rate $O(1/\sqrt{T})$ which requires a  total number of $O(T^3)$ and $O(T^2)$ calls to the approximation oracles of $\mX$ and $\mY$, respectively.
\end{enumerate}
\end{theorem}

The fully-detailed description of this theorem appears in the sequel as Theorems \ref{theorem:app on x}, \ref{theorem:app on y}, and \ref{thoerem: app on xy}.

\section{OCO-based Approachability with an Infeasible Saddle-Point Oracle}
In this section we introduce a meta-algorithm for approachability. It is based on reducing approachability to online convex optimization, using the approach of \cite{shimkin2016online}. At the heart of the algorithm is the use of new oracle we introduce for solving certain saddle-point problems, with similar structure to \eqref{eq:saddlepoint}, but that are efficient to approximate (to arbitrary accuracy) using approximation oracles. Before we can present these in detail, we first need to briefly review some basic concepts in OCO (for an in-depth introductory text we refer the reader to \cite{DBLP:journals/corr/abs-1909-05207}), and then we review the \textit{extended approximation oracle}, which will be a central building block in our algorithms.
\subsection{Online Convex Optimization}
\begin{definition}\label{dif:regret}
Given a convex and closed set $\mK\subseteq\reals^n$ and a positive integer $T$, an OCO algorithm $\mA$ for $\mK$ receives a sequence of convex loss functions $(f_t)_{t=1}^T\subset{}\reals^n\rightarrow\reals$ and outputs a sequence $(x_t)_{t=1}^T\subseteq\mK$ so that $x_{t} \gets \mA(f_1,\dots,f_{t-1})$ for all  $t =2, \dots ,T$. The regret of $\mA$ is then defined as:
\begin{align*}
\text{Regret}_{\mA}(f_1,\dots,f_T) :=\sum_{t=1}^T f_t(x_t) - \min_{x \in \mK} \sum_{t=1}^Tf_t(x).
\end{align*}
\end{definition}

\begin{definition}[Online (sub)Gradient Descent (OGD)]
\label{dif:OGD}
Given a convex and compact set $\mK\subset\reals^n$ and a sequence of $T$ convex functions $(f_t)_{t=1}^T\subset\reals^n\rightarrow\reals$, OGD produces the sequence of iterates $(x_t)_{t=1}^T\subseteq\mK$ as follows:
\begin{align*}
\forall t\geq 1: \quad x_{t+1} \gets \Pi_{\mK}(x_t - \eta \partial{}f_t(x_t)),
\end{align*}
where $x_1$ is some arbitrary feasible point in $\mK$, $ \eta > 0 $ is a fixed step-size,  $ \Pi_{\mK}(z) = \arg\min_{x \in \mK} \|x - z\|_2^2 $ denotes the Euclidean projection onto $\mK$, and $\partial{}f_t(x_t)$ denotes some subgradient of the convex function $f_t$ at the point $x_t$.
\end{definition}
\begin{theorem}[Regret of OGD, see for instance \cite{DBLP:journals/corr/abs-1909-05207}] 
\label{theorem:OGD_gar}
Suppose OGD is applied with step-size  $\eta = \frac{D_{\mK}}{G \sqrt{T}}$ for $G \geq \max_{t\in[T]}\Vert{\partial{}f_t(x_t)}\Vert_2$.  Then, its regret satisfies the bound:
\begin{align*}
\text{Regret}_{OGD}(f_1,\dots,f_T) \leq \frac{3GD_{\mK}}{2} \sqrt{T}. 
\end{align*}
\end{theorem}

\subsection{The extended approximation oracle}
The extended approximation oracle, introduced originally in \cite{kakade2007playing}, allows to extend an approximation oracle  for a convex and compact set $\mK$ with ratio $\alpha_{\mK}$ (Definition \ref{def:app oracle}), into an oracle that can optimize optimally w.r.t. the set $\alpha_{\mK}\mK$, but produces points that are not guaranteed to be in $\mK$ nor $\alpha_{\mK}\mK$. 
\begin{lemma}[see \cite{garber2017efficient}] 
\label{extended oracle}
Let $\mK\subseteq\reals^n$ be a convex, compact and non-negative set and let $\mathcal{O}_{\mK}$ be a corresponding approximation oracle with ratio $\alpha_{\mK}$ (Definition \ref{def:app oracle}). Given $c \in\reals^n$ we write $c=c^+ +c^-$, where $c^+ $ equals to $c$ on all non-negative coordinates of $c$ and zero everywhere else and $c^- $ equals to $c$ on all negative coordinates of $c$ and zero everywhere else. The extended approximation oracle given by the mapping:
\begin{align}
\widehat{\oraclek}(c)=\left(v ,s \right):= \begin{cases}
    \left({\oraclek(c^+)-\alpha R_\mK \vec{c}^-,\oraclek(c^+)}\right) & \text{if } \alpha_{\mK} \geq 1; \\
   \left({\oraclek(-c^-)-R_\mK \vec{c}^+,\oraclek(-c^-)}\right) & \text{if } \alpha_{\mK} < 1,
\end{cases} 
\end{align}
where for any vector $q\in \mathbb{R}^n$ we denote $\vec{q} = q/\Vert{q}\Vert$ if $q \neq 0$, and $\vec{q} = 0$ otherwise.
The outputted pair $(v,s)$ satisfies the following three properties:
\begin{align*}
i.~\langle{v, c}\rangle \leq \min_{x \in \alpha_{\mK} \mK} \langle{x,c}\rangle, \qquad ii.~ s\leq v, \qquad iii.~\|v\| \leq (\alpha+2)R_\mK.
\end{align*}
\end{lemma}

\subsection{Approachability with Approximate Infeasible Saddle-Point Oracles}
Recall solving the saddle-point optimization problem in Eq. \eqref{eq:saddlepoint} allows to design an algorithm for approaching an approchable set $S$, however is not tractable in our setting. Instead, we propose to implement an oracle for solving a modified saddle-point optimization problem, one which involves two different Blackwell instances.
\begin{definition}[approximate infeasible saddle-point oracle]\label{def:sporacle}
Given a Blackwell instance $(\mX, \mY, \ell, S)$ and a modified instance $(\tilde{\mX} ,\tilde{\mY}, \ell, \tilde{S})$ with common bi-linear function $\ell$, we say $\oraclesp$ is an approximate infeasible saddle-point oracle, if there exists $R_{SP} > 0$ such that for any $w\in\mB$ and $\epsilon >0$,  $\mathcal{O}_{SP}(w,\epsilon)$ outputs a pair $(x,s)\in \mB(R_{SP})\times\mX$ satisfying the following two guarantees:
\begin{align*}
 i.~ \max_{y \in \mY} \langle w, \ell(x, y) \rangle \leq \min_{\tilde{x}  \in \tX} \max_{\tilde{y} \in \tY} \langle w, \ell(\tilde{x}, \tilde{y}) \rangle +\epsilon \qquad \textrm{and} \qquad
   ii.~ s \leq x.
\end{align*}
In particular, if $(\tilde{\mX} ,\tilde{\mY}, \ell, \tilde{S})$ is approachable, then $\max_{y \in \mY} \langle w, \ell(x, y) \rangle \leq h_{\tilde{S}}(w) + \epsilon$.
\end{definition}
In words: oracle $\oraclesp$ produces an infeasible (w.r.t. $\mX$) point $x$ whose guarantee against every action of the original adversary (with actions in $\mY$) is at least as good, up to $\epsilon$, as a min-max optimal action w.r.t. the modified Blackwell instance $(\tX,\tY,\ell,\tS)$. Since $x$ need not be feasible, the oracle also outputs a feasible action $s\in\mX$ which dominates $x$.

We are now ready to present our meta-algorithm for approachability which builds on two main pillars: I. the reduction of approachability to OCO, as established in \cite{shimkin2016online}, and II. the infeasible saddle-point oracle.
The following algorithm and Theorem \ref{theorem:oco-based app  alg} are based on the derivations in \cite{shimkin2016online}, when modified to work with our newly proposed  oracle.
\begin{algorithm}[H] 
\caption{OCO-Based Approachability with an Approximate Infeasible saddle-Point Oracle }
\begin{algorithmic}[1]
\label{oco-based app  alg}
\STATE \textbf{input:} $(\mX, \mY, \ell ,S), (\tilde{\mX}, \tilde{\mY}, \ell ,\tilde{S}), \oraclesp$ --- Blackwell instances and corresponding infeasible saddle-point oracle (Definition \ref{def:sporacle}), $\mA$ --- OCO algorithm for the unit ball $\mB$ (see Definition \ref{dif:regret}), $\epsilon >0$ --- oracle approx. error, $T$ --- number of rounds 
\FOR{$t = 1$ to $T$}
    \STATE $w_t \gets \mathcal{A}(f_1, \cdots, f_{t-1})$  \COMMENT{for $t=1$, $\mA$ outputs its initialization point in $\mB$}
    \STATE $ \left( x_t,s_t\right) \gets \oraclesp \left( w_t, \epsilon \right)$
    \STATE play $s_t$ and observe $y_t\in\mY$
    \STATE  set $f_t(w) :=h_{\tS}(w)-\langle w , \ell(x_t,y_t)\rangle$
\ENDFOR
\RETURN $\bar{\ell}_{x,T} = \frac{1}{T}\sum_{t=1}^{T}\ell(x_t,y_t)$, $\bar{\ell}_{s,T} = \frac{1}{T}\sum_{t=1}^{T}\ell(s_t,y_t)$
\end{algorithmic}
\end{algorithm}

\begin{theorem}\label{theorem:oco-based app  alg}
Suppose Algorithm \ref{oco-based app  alg} is applied w.r.t. two approachable Blackwell instances $(\mX, \mY, \ell ,S), (\tilde{\mX}, \tilde{\mY}, \ell ,\tilde{S})$, such that $(\mX,\mY,\ell,S)$ satisfies Assumption \ref{ass:1}, and $\tilde{S}$ is bounded. The following two guarantees hold:
\begin{align}\label{eq:oco-based:1}
i. ~d\left(\bar{\ell}_{x,T} ,  \tS \right) \leq \frac{\text{Regret}_{\mA}(f_1,\dots,f_T)}{T}+ \epsilon \quad \textrm{and} \quad ii. ~\bar{\ell}_{s,T} \leq \bar{\ell}_{x,T},
\end{align}

In particular, if OGD (Definition \ref{dif:OGD}) is used as the algorithm $\mA$ with step-size $\eta =\frac{2}{G\sqrt{T}}$, where %
$G=\max_{\tilde{s}\in \tS, x \in \mB(R_{SP}) ,y\in \mY}\Vert\tilde{s}-\ell(x,y) \Vert$, then
\begin{align*}
d\left(\bar{\ell}_{x,T} ,  \tS \right) \leq \frac{3 G}{\sqrt{T}}+ \epsilon.
\end{align*}
\end{theorem}
Theorem \ref{theorem:oco-based app  alg} guarantees on one hand, that the \textit{infeasible} average loss $\bar{\ell}_{x,T}$ approaches the modified target set $\tS$ and, on the other-hand, that it is dominated by the actual feasible average loss $\bar{\ell}_{s,T}$.
\begin{proof}
Using  Eq. \eqref{dif:dis} we have that,
\begin{align*}
d\left(\bar{\ell}_{x,T} , \tS \right) &=\max_{w: \|w\|\leq 1}\{\langle w , \bar{\ell}_{x,T} \rangle - h_{\tS}(w)\} = -\frac{1}{T}\min_{w: \|w\|\leq 1}\{\sum_{t=1}^T (h_{\tS}(w)-\langle w ,\ell  (x_t,y_t) \rangle\}.
\end{align*}
By the definition of the function $f_1,\dots,f_T$ in the algorithm we have that,
\begin{align*}
\text{Regret}_{\mA}(f_1,\dots,f_T)&=\sum_{t=1}^T  h_{\tS}(w_t)-\langle w_t , \ell(x_t,y_t)\rangle-\min_{w\in \mB}\sum_{t=1}^T  h_{\tS}(w)-\langle w , \ell (x_t,y_t)\rangle.
\end{align*}
Combining the two equations and using the definition of $\oraclesp(w,\epsilon)$ we have,
\begin{align}
\label{result1 oco-based app  alg}
d\left(\bar{\ell}_{x,T} , \tS \right) &\leq \frac{\text{Regret}_{\mA}(f_1,\dots,f_T)}{T}+\frac{1}{T}\sum_{t=1}^T \langle w_t , \ell (x_t,y_t)\rangle - h_{\tS}(w_t) \notag \\ 
&\leq \frac{\text{Regret}_{\mA}(f_1,\dots,f_T)}{T}+\frac{1}{T}\sum_{t=1}^T \max_{y\in \mY}\langle w_t , \ell (x_t,y)\rangle - h_{\tS}(w_t) \notag \\ 
&\underset{(1)}{\leq} \frac{\text{Regret}_{\mA}(f_1,\dots,f_T)}{T}+\frac{1}{T}\sum_{t=1}^T \min_{\tilde{x}\in\tX}\max_{\tilde{y}\in \tY}\langle w_t , \ell (\tilde{x},\tilde{y})\rangle - h_{\tS}(w_t) \notag \\ 
&\underset{(2)}{\leq} \frac{\text{Regret}_{\mA}(f_1,\dots,f_T)}{T}+\frac{1}{T}\sum_{t=1}^T \left( h_{\tS}(w_t)+\epsilon- h_{\tS}(w_t) \right) \notag \\
&\leq \frac{\text{Regret}_{\mA}(f_1,\dots,f_T)}{T}+\epsilon ,
\end{align}
where in (1) we have used the definition of $\oraclesp(w,\epsilon)$, and in (2) we used the fact that $(\tX,\tY,\ell ,\tS)$ is an approachable Blackwell instance (see Eq. \eqref{cond}). 

Thus, the first inequality in \eqref{eq:oco-based:1} holds. The second inequality in \eqref{eq:oco-based:1} holds since, according to the definition of $\oraclesp$, for all $t\in[T]$, $s_t \leq x_t$. By the facts that $s_t\in\mX\subset\reals^n_+, y_t\in\mY\subset\reals^m_+$, and $\ell(\cdot ,\cdot)$ is non-negative over $\reals^n_+\times\reals^m_+$ (Assumption \ref{ass:1}), it follows that  $\ell(s_t,y_t) \leq \ell (x_t,y_t)$ for all $t\in[T]$, and thus, $\bar{\ell}_{s,T} \leq \bar{\ell}_{x,T}$.

The second part of the theorem follows from plugging-in the regret bound in Theorem \ref{theorem:OGD_gar} together with the step-size listed in the theorem. This regret bound depends on some $G \geq \max_{t\in[T], w\in\mB}\Vert{\partial{}f_t(w)}\Vert$. Note that  $\partial{}f_t(w) = z_t - \ell(x_t,y_t)$ for some $z\in\tilde{S}$ (see also Eq. \eqref{support function}). Thus, the value of $G$ listed in the theorem is indeed suitable.
\end{proof}

\section{Approximation Oracle for the set $\mX$}\label{sec:x}
In this section we consider the scenario that (only) the player's set $\mX$ is accessible through an approximation oracle with ratio $\alpha_{\mX} \geq 1$, and we prove the first part of Theorem \ref{thm:main}.
Our construction will involve implementing the infeasible saddle-point oracle and applying it in Algorithm \ref{oco-based app  alg} w.r.t. the modified Blackwell instance: $(\tX,\tY,\ell,\tS) = (\alpha_{\mX}\mX, \mY, \ell, \alpha_{\mX}S)$.

First, we bring the following simple lemma that will allow us to simplify our presentation. The proof is given in the appendix.
\begin{lemma}[structure of bi-linear function]\label{bi-linear lemma}
Given a bi-linear function $\ell:\reals^n\times\reals^m\rightarrow\reals^d$, there exists matrices $L_1,\dots,L_d\in\reals^{n\times m}$ such that
\begin{align}\label{bi-linear form}
\forall x,y:\quad [\ell (x,y)]_i=x^\top L_i y, \quad i \in [d].
\end{align}
Consequently, for any $w\in\mB$, the matrix $A_{\ell,w} := \sum_{i=1}^dw(i)L_i$ satisfies
\begin{align*}
\langle{w,\ell(x,y)}\rangle = x^{\top}A_{\ell,w}y, \quad \textrm{and} \quad \Vert{A_{\ell,w}}\Vert \leq \Vert{\ell}\Vert,
\end{align*}
where we define $\Vert{\ell}\Vert := \max_{u\in\mB}\Vert{\sum_{i=1}^du(i)L_i}\Vert$.
\end{lemma}

Algorithm \ref{OCO-SP_x} implements an infeasible saddle-point oracle by building on an original OCO-based primal-dual method in which, given some $w\in\mB$, the bi-linear function $g(x,y) = \langle{w,\ell(x,y)}\rangle$ (as in \eqref{eq:saddlepoint}) is optimized over $\mY$ via an OCO algorithm, and \textit{infeasibly} optimized over $\alpha_{\mX}\mX$ using the extended approximation oracle.
\begin{algorithm}[H]
\begin{algorithmic}[1]
\caption{AISPOX - Approximate Infeasible SP Oracle using approximation oracle for $\mX$}
\label{OCO-SP_x}
\STATE \textbf{input:} $(g,\oraclex ,N,\mA_{\mY})$, where  $g(x,y):=x^\top A y$, $A\in \reals^{n\times m} $, $\oraclex$ is an approximation oracle for set $\mX$ (Definition \ref{def:app oracle}), $N$ is number of iterations, $\mathcal{A}_\mY$ is an OCO algorithm for the set $\mY$ (Definition \ref{dif:regret})
\FOR{$i=1,2,\dots N$}
    \STATE $y_{i} \gets \mathcal{A}_\mY (g_1,\ldots g_{i-1})$ \COMMENT{if $i=1$, $y_{i}$ is the initialization point of $\mA_{\mY}$} 
    \STATE $(x_i,s_i) \gets \widehat{\oraclex}(A y_i)$ \COMMENT{call to extended approximation oracle (Lemma \ref{extended oracle})}
    \STATE Set $g_i(y):=-g(x_i,y)=-{x_i}^\top A y$
\ENDFOR
\RETURN{$\left({\bar{x}:=\frac{1}{N}\sum_{i=1}^N x_i,\bar{s}:=\frac{1}{N}\sum_{i=1}^N s_i}\right)$}
\end{algorithmic}
\end{algorithm}

\begin{lemma}\label{lemma:OCO-SP_x}
Let $\mX\subset \reals^n_+ , \mY\subset \reals^m_+$ be compact convex sets and let $g:\reals^n \times \reals^m \to \reals$ be a bi-linear function given as $g(x,y)=x^\top A y$. Algorithm \ref{OCO-SP_x} has the following guarantees:
\begin{align*}
&i.~\max_{y \in \mY} g(\bar{x},y) \leq \min_{x\in \alpha_{\mX} \mX} \max_{y \in \mY} g(x,y)+ \frac{\text{Regret}_{\mathcal{A}_\mY}(g_1,\dots,g_N)}{N},\quad  ii.~ \bar{s}\leq \bar{x}, \\
&iii.~\Vert{\bar{x}}\Vert \leq (\alpha_{\mX}+2)R_{\mX}.
\end{align*}
In particular, if OGD (Definition \ref{dif:OGD}) is used as  $\mA_{\mY}$ with step-size $\eta =\frac{D_{\mY}}{G_{\mY}\sqrt{N}}$, where $G_{\mY}=(\alpha_{\mX}+2)\Vert \ell \Vert R_\mX$, and $A = A_{\ell,w} := \sum_{i=1}^dw(i)L_i$ for some $w\in\mB$ and the matrices $L_1,\dots,L_d$ defined in Lemma \ref{bi-linear lemma}, then 
\begin{align*}
\max_{y \in \mY} \langle{w,\ell(\bar{x},y)}\rangle \leq \min_{x\in \alpha_{\mX} \mX} \max_{y \in \mY}\langle{w,\ell(x,y)}\rangle+ \frac{3 D_{\mY} G_{\mY}}{2\sqrt{N}}.
\end{align*}
\end{lemma}
\begin{proof}
Using the bi-linearity of $g$ and the regret guarantee of $\mathcal{A}_\mY$ we have that,
\begin{align*}
\max_{y \in \mY}\ g(\bar{x},y) = \max_{y \in \mY}\frac{1}{N}\sum_{i=1}^N g(x_i,y) \leq \frac{1}{N}\sum_{i=1}^N g(x_i,y_i) +\frac{\text{Regret}_{\mathcal{A}_\mY}(g_1,\dots,g_N)}{N}.
\end{align*}
Since each $x_i$ is the output of the oracle $\widehat{\oraclex}$ (Lemma \ref{extended oracle}) we have that,
\begin{align} \label{result1 OCO-SP_x}
\max_{y \in \mY} g(\bar{x},y) &{\leq} \frac{1}{N}\sum_{i=1}^N\min_{x\in \alpha_{\mX} \mX} g(x,y_i)+\frac{\text{Regret}_{\mathcal{A}_\mY}(g_1,\dots,g_N)}{N} \notag \\
&\underset{(1)}{\leq}  \min_{x\in \alpha_{\mX} \mX} g(x,\bar{y})+\frac{\text{Regret}_{\mathcal{A}_\mY}(g_1,\dots,g_N)}{N} \notag \\
&\underset{(2)}{\leq}  \min_{x\in \alpha_{\mX} \mX} \max_{y\in \mY} g(x,y)+\frac{\text{Regret}_{\mathcal{A}_\mY}(g_1,\dots,g_N)}{N} ,
\end{align}
where (1) follows from the bi-linearity of $g$ and (2) follows since $\bar{y}\in\mY$.

Thus, the first inequality in the first part of the lemma holds. The second inequality ($\bar{s} \leq \bar{x}$) and the third one (upper-bound $\Vert{\bar{x}}\Vert$)  follow immediately from the definition of the extended approximation oracle  (Lemma \ref{extended oracle}).

The second part of the Lemma follows from combining \eqref{result1 OCO-SP_x} with Lemma \ref{bi-linear lemma} and with the regret guarantee of OGD given in Theorem \ref{theorem:OGD_gar}. In particular, note that using  Lemma \ref{bi-linear lemma} we have that for each $i\in[N]$, $\Vert{\partial{}g_i(y_i)}\Vert = \Vert{A_{\ell,w}^{\top}x_i}\Vert \leq \Vert{\ell}\Vert\Vert{x_i}\Vert \leq \Vert{\ell}\Vert(\alpha_{\mX}+2)R_{\mX}$, and thus our choice of constant $G_{\mY}$ and step-size $\eta$ are compatible with those in Theorem \ref{theorem:OGD_gar}.
\end{proof}

We are now ready to state in full detail and prove the first part of Theorem \ref{thm:main}.
\begin{theorem}\label{theorem:app on x} 
Let $(\mX,\mY,\ell,S)$ be a Blackwell instance that satisfies Assumption \ref{ass:1}. Consider running Algorithm \ref{oco-based app  alg} w.r.t. $(\mX,\mY,\ell,S)$ and the modified Blackwell instance:
$(\tilde{\mX}, \tilde{\mY}, \ell, \tilde{S}) := (\alpha_{\mX}\mX, \mY, \ell, \alpha_{\mX}S)$,
with the following implementation:
\begin{itemize}
\item
OGD with step-size $\eta= \frac{2}{G_1\sqrt{T}}$ is the OCO algorithm for the ball $\mB$.
\item
On each iteration $t\in[T]$, the call to the saddle-point oracle $(x_t,s_t)\gets \oraclesp(w_t, \epsilon)$ is implemented using Algorithm \ref{OCO-SP_x} as follows:
\begin{align*}
\left(x_t , s_t \right) \gets \textsc{AISPOX}  \left( g_t(x,y) :=x^\top \left( \sum_{i=1}^d {w_t}(i)L_i \right) y, \oraclex, N(T) , OGD_{\mY}  \right),
\end{align*}
where $N(T)=G_{\mY}^2 D_{\mY}^2 T$, and $OGD_{\mY}$ denotes the use of OGD as the OCO within Algorithm  \ref{OCO-SP_x}  with step-size $\mu= \frac{D_{\mY}}{G_{\mY} \sqrt{N(T)}}$ ($G_{\mY}$ is as in Lemma \ref{lemma:OCO-SP_x}).
\end{itemize}
Then, 
\begin{align}\label{eq:s_t conv x}
d\left(\bar{\ell}_{s,T} , {\left( \alpha_{\mX} S\right)}_{\downarrow}\right)\leq\frac{6G_1+3}{2\sqrt{T}},
\end{align}
and the overall number of calls to the approximation oracle of $\mX$ is $O( G_{\mY}^2 D_{\mY}^2T^2)$, where 
$G_1:=\max_{s\in S, x \in \mB((\alpha_{\mX} +2)R_{\mX}) ,y \in \mY} \|\alpha_{\mX} s -\ell (x,y)\|$.

\end{theorem}
\begin{proof}
First, we note that the Blackwell instance $(\tX, \tY, \ell, \tS)$ is indeed approachable (see Lemma \ref{lemma: alphaS app} in the appendix).

Denote $\epsilon =\frac{3}{2\sqrt{T}}$. By Applying Lemma \ref{lemma:OCO-SP_x} with the choice of $N(T)$, we have that on each iteration $t\in[T]$ of Algorithm \ref{oco-based app  alg}, $(x_t,s_t)$ is a valid output of an oracle $\oraclesp$ w.r.t. the Blackwell instances listed in the theorem and with parameter $R_{SP} = (\alpha_{\mX}+2)R_{\mX}$, when called with input $(w_t, \epsilon)$. Thus, applying Theorem \ref{theorem:oco-based app  alg} we have that,
\begin{align*}
d\left(\bar{\ell}_{x,T} , \alpha_{\mX} S\right) \leq \frac{3G_1}{\sqrt{T}} + \epsilon = \frac{6G_1+3}{2\sqrt{T}}.
\end{align*} 

Now, define $\ell_{T} := \argmin_{s \in \tS} \|\bar{\ell}_{x,T} - s\|$, $u_{T}:=\bar{\ell}_{x,T} - \ell_{T}$, and $\tilde{\ell}_{T}:=\bar{\ell}_{s,T} - u_{T}$. First, note that $\tilde{\ell}_T = \bar{\ell}_{s,T} - \bar{\ell}_{x,T} + \ell_T \leq \ell_T$ (since $\bar{\ell}_{s,T} \leq \bar{\ell}_{x,T}$), and since $\ell_T\in\alpha_{\mX}S$, we have that $\tilde{\ell}_T\in(\alpha_{\mX}S)_{\downarrow}$. Thus, we have that
\begin{align*}
d\left(\bar{\ell}_{s,T} ,{\left( \alpha_{\mX} S\right)}_\downarrow \right) \leq \|\bar{\ell}_{s,T}-\tilde{\ell}_{T}\| =\|u_{T}\| = d\left(\bar{\ell}_{x,T} , \alpha_{\mX} S\right),
\end{align*}
and the first part of the theorem holds.

The bound on the total number of calls to the approximation oracle $\oraclex$ follows since each iteration of Algorithm \ref{oco-based app  alg} makes a single call to  Algorithm \ref{OCO-SP_x}, which in turn makes $N(T)$ calls to $\oraclex$, and so the total number of calls is given by $T\cdot{}N(T)$.
\end{proof}

\section{Approximation Oracle for the set $\mY$}\label{sec:y}
In this section we consider the scenario that (only) the adversary's set $\mY$ is accessible through an approximation oracle with ratio $1 \geq \alpha_{\mY} > 0$, and we prove the second part of Theorem \ref{thm:main}. Our construction will involve implementing the infeasible saddle-point oracle and applying it in Algorithm \ref{oco-based app  alg} w.r.t. the modified Blackwell instance: $(\tX,\tY,\ell,\tS) = (\mX, \alpha_{\mY}^{-1}(\mY-\mB_+(R_{\mY})), \ell, \alpha_{\mY}^{-1}(S-\mB_+(\tilde{R}))$, where $\tilde{R}$ shall be specified later on. Note that here, as opposed to Section \ref{sec:x}, the sets $\tY,\tS$ are not only scaled, but first shifted (by $-\mB_+(R)$ and $-\mB_+(\tilde{R})$, respectively). This is due to the fact that Definition \ref{def:sporacle} is not symmetric w.r.t. the $x$ and $y$ variables, and will be more apparent in the proofs.

The following Algorithm \ref{OCO-SP-y} and Lemma \ref{lemma:OCO-SP-y } are analogues to Algorithm \ref{OCO-SP_x} and Lemma \ref{lemma:OCO-SP_x} from the previous section, with the modification than now the $x$ variable is updated via an OCO algorithm, and the $y$ variable is updated via the extended approximation oracle.
\begin{algorithm}[H]
\begin{algorithmic}[1]
\caption{AISPOY - Approximate Infeasible SP Oracle using approximation Oracle for $\mY$}\label{OCO-SP-y}
\STATE \textbf{input:} $(g,\oracley ,N,\mA_{\mX})$, where $g(x,y)=x^\top A y$ ,$A\in \reals^{n\times m}$, $\oracley$ is approx. oracle for $\mY$, $N$ is  number of iterations, $\mathcal{A}_{\mX}$ is an OCO algorithm for $\mX$
\FOR{$i=1,2,\dots N$}
   \STATE $x_{i}\gets \mathcal{A}_\mX (g_1,\ldots g_{i-1})$ \COMMENT{if $i=1$, $x_{i}$ is the initialization point of $\mA_{\mX}$} 
    \STATE $(y_i,s_i) \gets \widehat{\oracley}(-x_i^\top A)$ \COMMENT{call to extended approximation oracle (Lemma \ref{extended oracle})}
    \STATE set $g_i(x):=g(x,y_i)=x^\top A {y_i}$
\ENDFOR
\RETURN $\left( \bar{x}:=\frac{1}{N}\sum_{i=1}^N x_i, \bar{x}\right)$ \COMMENT{$\bar{x}$ sent twice to fit Definition \ref{def:sporacle}}
\end{algorithmic}
\end{algorithm}

\begin{lemma} \label{lemma:OCO-SP-y }
Let $\mX\subset \reals^n_+ , \mY\subset \reals^m_+$ be compact convex sets and let $g:\reals^n \times \reals^m \to \reals$ be a bi-linear function given as $g(x,y)=x^\top A y$. Algorithm \ref{OCO-SP-y} guarantees that
\begin{align}
\label{result1 OCO-SP-y}
\max_{y \in  \mY} g(\bar{x},y) \leq \min_{x\in \mX} \max_{y \in \alpha_{\mY}^{-1}(\mY- \mathcal{B}_+(R_\mY))} g(x,y)+ \frac{\text{Regret}_{\mathcal{A}_\mX}(g_1,\dots,g_N)}{\alpha_{\mY} N}.
\end{align}

In particular, if OGD (Definition \ref{dif:OGD}) is used as $\mA_{\mX}$ with step-size $\eta =\frac{D_{\mX}}{G_{\mX}\sqrt{N}}$, where $G_{\mX}:=(\alpha_\mY+2) \Vert \ell \Vert R_\mY$, and $A = A_{\ell,w} := \sum_{i=1}^dw(i)L_i$ for some $w\in\mB$ and the matrices $L_1,\dots,L_d$ defined in Lemma \ref{bi-linear lemma}, then 
\begin{align*}
\max_{y \in \mY} \langle{w,\ell(\bar{x},y)}\rangle \leq \min_{x\in \mX} \max_{y \in\alpha_{\mY}^{-1}(\mY- \mathcal{B}_+(R_\mY))} \langle{w,\ell(x,y)}\rangle + \frac{3 D_{\mX} G_{\mX}}{2\alpha_{\mY}\sqrt{N}}.
\end{align*}
\end{lemma}
\begin{proof}
Recalling that $g$ is bi-linear we have that,
\begin{align*}
\max_{y \in \alpha_{\mY} \mY } g(\bar{x},y) &= \max_{y \in \alpha_{\mY} \mY } \frac{1}{N}\sum_{i=1}^Ng(x_i,y) \underset{(1)}{\leq} \frac{1}{N}\sum_{i=1}^N g(x_i,y_i) \notag \\
&\underset{(2)}{\leq} \min_{x \in \mX}\frac{1}{N}\sum_{i=1}^N g(x,y_i)+\frac{\text{Regret}_{\mathcal{A}_\mX}(g_1,\dots,g_N)}{N} \notag \\
&= \min_{x \in \mX} g\left({x,\frac{1}{N}\sum_{i=1}^Ny_i}\right)+ \frac{\text{Regret}_{\mathcal{A}_\mX}(g_1,\dots,g_N)}{N}
\notag \\
&\underset{(3)}{\leq} \min_{x \in \mX}  \max_{y\in \mY-\mB_+(R_{\mY})} g(x,y)+ \frac{\text{Regret}_{\mathcal{A}_\mX}(g_1,\dots,g_N)}{N},
\end{align*}
where (1) follows since each $y_i$ is the output of the extended approximation oracle  (Lemma \ref{extended oracle}), (2) follows from the regret guarantee of $\mA_{\mX}$, and (3) follows due to the definition of the extended approximation oracle (Lemma \ref{extended oracle}), which guarantees that $\frac{1}{N}\sum_{i=1}^Ny_i$ is in $\mY-\mB_+(R_{\mY})$.

Using again the bi-linearity of $g$, we thus have that
\begin{align*}
\max_{y \in \mY}g(\bar{x},y) &= \frac{1}{\alpha_{\mY}}\max_{y \in \alpha_{\mY} \mY}g(\bar{x},y)\\ 
&\leq  \frac{1}{\alpha_{\mY}}\left(\min_{x\in \mX} \max_{y \in \mY-\mB_+(R_{\mY}) } g(x,y)+ \frac{\text{Regret}_{\mathcal{A}_\mX}(g_1,\dots,g_N)}{N}\right) \\
 &=\frac{\alpha_{\mY}}{\alpha_{\mY}}\min_{x\in \mX} \max_{y \in \alpha_{\mY}^{-1}(\mY-\mB_+(R_{\mY})) } g(x,y)+\frac{\text{Regret}_{\mathcal{A}_\mX}(g_1,\dots,g_N)}{\alpha_{\mY} N} \\
&=\min_{x\in \mX} \max_{y \in \alpha_{\mY}^{-1}(\mY-\mB_+(R_{\mY})) } g(x,y)+ \frac{\text{Regret}_{\mathcal{A}_\mX}(g_1,\dots,g_N)}{\alpha_{\mY} N},
\end{align*}
which proves the first part of the lemma.

The proof of the second part of the lemma (implementation with OGD) follows the exact same lines as in the proof of Lemma \ref{lemma:OCO-SP_x}.
\end{proof}

We  now state in full detail and prove the second part of Theorem \ref{thm:main}.
\begin{theorem}\label{theorem:app on y}
Let $(\mX,\mY,\ell,S)$ be a Blackwell instance that satisfies Assumption \ref{ass:1}. Consider running Algorithm \ref{oco-based app  alg} w.r.t. $(\mX,\mY,\ell,S)$ and the modified Blackwell instance:
($\tilde{\mX}, \tilde{\mY}, \ell, \tilde{S}) = \left({\mX, \frac{1}{\alpha_{\mY}}(\mY-\mB_+(R_{\mY})), \ell, \frac{1}{\alpha_{\mY}}(S-\mB_{+}(\tilde{R}))}\right)$,
where  $\tilde{R}=\Vert \ell \Vert R_\mX R_\mY$, with the following implementation:
\begin{itemize}
\item
OGD is the OCO algorithm for the ball $\mB$ and it is used with step-size $\eta= \frac{2}{G_3 T}$.
\item
On each iteration $t\in[T]$, the call to the saddle-point oracle $(x_t,s_t)\gets \oraclesp(w_t, \epsilon)$ is implemented using Algorithm \ref{OCO-SP-y} as follows:
\begin{align*}
\left(x_t , s_t \right) \gets \textsc{AISPOY}  \left( g_t(x,y) :=x^\top \left( \sum_{i=1}^d {w_t}(i)L_i \right) y, \oracley, N(T) , OGD_{\mX}  \right),
\end{align*}
where $N(T):=\frac{G_{\mX}^2 D_{\mX}^2 T}{{\alpha_{\mY}}^2}$, and $OGD_{\mX}$ denotes the use of OGD as the online Algorithm within Algorithm  \ref{OCO-SP-y}  with step-size $\mu= \frac{D_{\mX}}{G_{\mX} \sqrt{N(T)}}$, where $G_{\mX}$ is as defined in Lemma \ref{lemma:OCO-SP-y }.
\end{itemize}
Then, 
\begin{align}\label{eq:s_t conv y}
d\left(\bar{\ell}_{s,T} , {\left(\alpha_{\mY}^{-1}S\right)}_{\downarrow}\right) \leq \frac{6G_3+3}{2\sqrt{T}},
\end{align}
and the overall number of calls to the approximation oracle of $\mY$  is $O\left({ \frac{G_{\mX}^2 D_{\mX}^2 T^2}{\alpha_{\mY}^2}}\right)$, where 
$G_3:=\max_{s\in S,u \in \mB_+(\Vert \ell \Vert R_\mX R_\mY),x \in \mX ,y \in \mY} \left\Vert\frac{1}{\alpha_{\mY}}(s- u) - \ell (x,y)\right\Vert$.
\end{theorem}
\begin{proof}

First, we note that the Blackwell instance $(\tX, \tY, \ell, \tS)$ is indeed approachable (see Lemma \ref{lemma:y_cond} in the appendix).

Denote $\epsilon =\frac{3}{2\sqrt{T}}$. By Applying Lemma \ref{lemma:OCO-SP_x} with our choice of $N(T)$, we have that on any iteration $t\in[T]$ of Algorithm \ref{oco-based app  alg}, $(x_t, s_t) = (x_t,x_t)$ (recall Algorithm \ref{OCO-SP-y} returns the same point for both outputs) is a valid output of an oracle $\oraclesp$ w.r.t. the Blackwell instances listed in the theorem and with parameter $R_{SP} = R_{\mX}$, when called with input $(w_t, \epsilon)$. Thus, applying Theorem \ref{theorem:oco-based app  alg} we have that,
\begin{align*}
d\left(\bar{\ell}_{x,T} , \tilde{S}\right) \leq \frac{3G_3}{\sqrt{T}} + \epsilon = \frac{6G_3+3}{2\sqrt{T}},
\end{align*} 
which, by recalling that in this case $\bar{\ell}_{x,T} = \bar{\ell}_{s,T}$ (since $x_t=s_t$ for all $t$), and our choice of $\tilde{S}$,  implies that 
\begin{align*}
d\left(\bar{\ell}_{s,T} , \frac{1}{\alpha_{\mY}}\left({S-\mB_+(\tilde{R})}\right)\right) \leq  \frac{6G_3+3}{2\sqrt{T}}.
\end{align*} 
Since $\frac{1}{\alpha_{\mY}}\left({S-\mB_+(\tilde{R})}\right) \subseteq \left({\frac{1}{\alpha_{\mY}}S}\right)_{\downarrow}$, the first part of the theorem follows.

As in the proof of Theorem \ref{theorem:app on x}, a straightforward accounting yields that the overall number of calls to the approximation oracle $\mY$ is $T\cdot{}N(T)$.
\end{proof}

\section{Approximation Oracles for both $\mX$ and $\mY$}
In this section we consider the scenario that both the player's set $\mX$ and the adversary's set $\mY$ are accessible through approximation oracles with ratios $\alpha_{\mX} \geq 1$ and $1 \geq \alpha_{\mY} > 0$, respectively, and we prove the third part of Theorem \ref{thm:main}.

Our construction will naturally involve implementing the infeasible saddle-point oracle and applying it in Algorithm \ref{oco-based app  alg}, by combining the ideas from Section \ref{sec:x} and Section \ref{sec:y}, i.e., we shall consider the modified Blackwell instance: $(\tX,\tY,\ell,\tS) = (\alpha_{\mX}\mX, \alpha_{\mY}^{-1}(\mY-\mB_+(R_{\mY})), \ell, \alpha_{\mX}\alpha_{\mY}^{-1}(S-\mB_+(\tilde{R}))$, where $\tilde{R}$ is as in Section \ref{sec:y}.

\subsection{Implementing the saddle-point oracle}
When both sets $\mX,\mY$ are accessible through approximation oracles, we can no longer implement the infeasible saddle-point oracle using (relatively) simple implementations as in Algorithms \ref{OCO-SP_x} and \ref{OCO-SP-y}, which relied on running an off-the-shelf OCO algorithm for one of the sets. Instead, we shall derive an OCO algorithm which can \textit{infeasibly} optimize w.r.t. $\alpha_{\mX}\mX$. This algorithm will be a simple adaptation of OGD that works with so-called \textit{infeasible projections} and is adapted from \cite{garber2017efficient}.  The algorithm is given as Algorithm \ref{alg:infeasibile OGD} below.
Before we can present the algorihtm, we first present the main component in it --- the Frank-Wolfe method for computing infeasible projections, which is adapted from \cite{kakade2007playing}.

\begin{algorithm}[H]
\caption{Infeasible Projection onto $\alpha_{\mX}\mX$ via Frank-Wolfe}
\label{alg:infeasible-proj}
\begin{algorithmic}[1]
\STATE \textbf{input:} $(y,\epsilon, \oraclex)$, where  $y\in\reals^n$ is point to project, $\epsilon>0$ is approx. error, and $\oraclex$ is approximation oracle for $\mX$ 
\STATE Set $x_1=\tilde{s}_1 \gets$ some arbitrary point in $\mX$
\STATE $ \tilde{y} \gets \min\{ 1 ,\frac{\alpha_{\mX} R_{\mX}}{\Vert y \Vert}\} \cdot y$ \COMMENT{projection of $y$ onto the ball $\mB (\alpha_{\mX} R_{\mX})$}
\FOR{$t=1,2,\ldots$}
   \STATE $\left(v_t,s_t\right) \gets \widehat{\oraclex}(x_t -\tilde{y})$ \COMMENT{$x_t-\tilde{y}$ is gradient of $f(x):=\frac{1}{2}\Vert{x-\tilde{y}}\Vert^2$ at $x_t$}
   \IF{$\langle{x_t-\tilde{y}, x_t-v_t}\rangle\leq \epsilon$}
      \RETURN $(\tilde{s}_t,x_t)$
   \ELSE 
      \STATE $\lambda_t \gets\arg\min_{\lambda \in (0,1]} \|(1-\lambda)\cdot x_t + \lambda\cdot v_t -y\|^2$
      \STATE $x_{t+1}\gets(1-\lambda_t)\cdot x_t + \lambda_t\cdot v_t$
      \STATE $\tilde{s}_{t+1}\gets(1-\lambda_t)\cdot \tilde{s}_t + \lambda_t\cdot s_t$
  \ENDIF
\ENDFOR
\end{algorithmic}
\end{algorithm}

\begin{lemma}
\label{lemma: infesProj}
Given $y\in\reals^n$ and $\epsilon \in (0,4{(\alpha_\mX +2)}^2 {R_\mX}^2]$, Algorithm \ref{alg:infeasible-proj} stops after at most $T^* = O \left( \frac{{(\alpha_\mX +2)}^4 {R_\mX}^4}{\epsilon^2}\right)$ iterations, and its output, the tuple $(x_{T^*},\tilde{s}_{T^*})\in\mB((\alpha_{\mX}+2)R_{\mX})\times\mX$ satisfies the following two guarantees: 
\begin{align}\label{eq:infesProj:1}
i.~\forall z \in \alpha_\mX \mX: ~~ {\|z-x_{T^*}\|}^2 \leq {\|z-y\|}^2 +2\epsilon \quad \textrm{and} \quad ii.~\tilde{s}_{T^*} \leq x_{T^*}.
\end{align}
\end{lemma}
\begin{proof}
Fix some $z \in \alpha_\mX \mX$. It holds that,
\begin{align*}
{\|z-\tilde{y}\|}^2 &= {\|(z-x_{T^*})+(x_{T^*}-\tilde{y})\|}^2 \\
&= {\|x_{T^*}-\tilde{y}\|}^2+{\|z-x_{T^*}\|}^2 - 2\langle{x_{T^*}-z, x_{T^*}-\tilde{y}}\rangle \\
&\geq {\|z-x_{T^*}\|}^2 - 2\langle{x_{T^*}-z, x_{T^*}-\tilde{y}}\rangle \\
&\underset{(1)}{\geq} {\|z-x_{T^*}\|}^2 - 2\langle{x_{T^*}-v_{T^*}, x_{T^*}-\tilde{y}}\rangle \\
&\underset{(2)}{\geq} {\|z-x_{T^*}\|}^2 - 2\epsilon,
\end{align*}
where in (1) we use the guarantees of the extended approximation oracle (Lemma \ref{extended oracle}), and in (2) we use the fact that $\langle{x_{T^*}-\tilde{y}, x_{T^*}-v_{T^*}}\rangle\leq \epsilon$ (the stopping condition).

Thus, by the definition of $\tilde{y}$, we have that
\begin{align*}
{\|z-x_{T^*}\|}^2 \leq {\|z-\tilde{y}\|}^2 + 2\epsilon \leq {\|z-y\|}^2 + 2\epsilon.
\end{align*}
Additionally, the inequality $\tilde{s}_{T^*} \leq x_{T^*}$ follows immediately from the guarantees of the extended approximation oracle. Thus, we have proved \eqref{eq:infesProj:1}. 

Observe that the guarantee $(x_{T^*},\tilde{s}_{T^*})\in\mB((\alpha_{\mX}+2)R_{\mX})\times\mX$ also follows directly from the guarantees of the extended approximation oracle.

We continue to prove the bound on the number of iterations $T^*$. For any $t: 1\leq t< T^*$, by the choice of $\lambda_t$, we have that for any $\lambda\in[0,1]$,
\begin{align*}
{\|x_{t+1}-\tilde{y}\|}^2 &= \|(1-\lambda_t)\cdot x_t + \lambda_t\cdot v_t-\tilde{y}\|^2 \leq  \|(1-\lambda)\cdot x_t + \lambda\cdot v_t-\tilde{y}\|^2 \\
&=  \|x_t-\tilde{y} +\lambda( v_t-x_t)\|^2\\
&=\|x_t-\tilde{y} \|^2+\lambda^2\|v_t-x_t\|^2 - 2\lambda\langle{x_t-\tilde{y}, x_t-v_t}\rangle \\
&\underset{(1)}{\leq}  \|x_t-\tilde{y} \|^2 +2{\lambda}^2(\|v_t\|^2+\|x_t\|^2)-2\lambda\epsilon \\
&\underset{(2)}{\leq} \|x_t-\tilde{y} \|^2 +4{\lambda}^2({(\alpha_\mX +2)}^2 {R_\mX}^2)-2\lambda\epsilon,
\end{align*}
where  (1) holds since the stopping condition is not satsified on iteration $t < T^*$, and (2) holds since  $x_t,v_t \in \mB((\alpha_{\mX} +2) R_{\mX})$ . 

In particular,  choosing $\lambda=\frac{\epsilon}{4{(\alpha_\mX +2)}^2 {R_\mX}^2}$ (which is in $[0,1]$ due to our assumption on $\epsilon$), we get
\begin{align*}
{\|x_{t+1}-\tilde{y}\|}^2 &\leq \|x_t-\tilde{y}\|^2 -\frac{\epsilon^2}{4{(\alpha_\mX +2)}^2 {R_\mX}^2} \leq \ldots \leq \|x_1-\tilde{y}\|^2 -t\frac{\epsilon^2}{4{(\alpha_\mX +2)}^2 {R_\mX}^2}.
\end{align*}
Since $x_1\in\mB(R_{\mX})$ and $\tilde{y}\in\mB(\alpha_{\mX}R_{\mX})$ we conclude that
\begin{align*}
T^* \leq  1+ \frac{2{(1+\alpha_{\mX}^2)} {R_{\mX}}^2}{\frac{\epsilon^2}{4{(\alpha_{\mX} +2)}^2 {R_{\mX}}^2}}= O\left({\frac{{(\alpha_{\mX} +2)}^4 {R_{\mX}}^4}{\epsilon^2}}\right).
\end{align*}

\end{proof}

\begin{algorithm}[H]
\caption{OGDWOF - Online Gradient Decent Without Feasibility}
\label{alg:infeasibile OGD}
\begin{algorithmic}[1]
\STATE \textbf{input:} $(N, \mu ,\xi)$, where $N$ is number of rounds, $\mu >0$ is the step-size, and $\xi>0$ is projection error
\STATE set $s_1 \gets $ some point in $\mX$ and $x_1 \gets \alpha_{\mX}s_1$
\FOR{$t=1.\ldots ,N$}
   \STATE receive loss vector $c_t\in\reals^n$
   \STATE $y_{t+1} \gets x_t -\mu c_t$
   \STATE $(s_{t+1},x_{t+1}) \gets$  output of Algorithm \ref{alg:infeasible-proj} with input $(y_{t+1},\xi)$
   \ENDFOR
\end{algorithmic}
\end{algorithm}

The proof of the following lemma follows closely the well-known proof for the regret bound of OGD \cite{zinkevich2003online}, and is given in the appendix for completeness.
\begin{lemma}
\label{lemma:infeasibile OGD}
Suppose Assumption \ref{ass:1} holds. The sequences $(x_t)_{t=1}^N\subset\mB((\alpha_{\mX}+2)R_{\mX})$, $(s_t)_{t=1}^N\subset\mX$ produced by Algorithm \ref{alg:infeasibile OGD} satisfy the following two guarantees:
\begin{align*}
&i.~ \forall x \in \alpha_\mX \mX: \quad \sum_{t=1}^N \langle x_t,c_t \rangle -\sum_{t=1}^N \langle x,c_t \rangle \leq  \frac{\alpha_{\mX}^2R_{\mX}^2}{\mu }  +\frac{\mu }{2} \sum_{t=1}^N \|c_t\|^2 + 2\xi N, \\
&ii.~  \frac{1}{N}\sum_{t=1}^Ns_t \leq \frac{1}{N}\sum_{t=1}^N x_t.
\end{align*}
\end{lemma}

We are now ready to present our implementation of the infeasible saddle-point oracle. The algorithm is a simple modification of  Algorithm \ref{OCO-SP-y}, when using Algorithm \ref{alg:infeasibile OGD} as the online algorithm to update the $x$ variable.
 \begin{algorithm}[H]
\begin{algorithmic}[1]
\caption{AISPOYX - Approximate Infeasible SP Oracle using approximation oracles for $\mX$ and $\mY$ }\label{OCO-SP-xy}
\STATE \textbf{Input:} $(g,\oraclex, \oracley ,N, \mu, \xi)$, where $g(x,y)=x^\top A y$ ,$A\in \reals^{n\times m}$, $\oraclex, \oracley$ are approximation oracles for $\mX$ and $\mY$ respectively, $N$ is number of iterations, $\mu > 0$ is the step-size for Algorithm  OGDWOF, $\xi$ is projection error 
\STATE Run AISPOY (Algorithm \ref{OCO-SP-y}) with the following changes: 
\begin{enumerate}
\item
Use OGDWOF with parameters $\mu,\xi$ as the OCO algorithm $\mA_{\mX}$
\item
Given $(x_i)_{i=1}^N, (s_i)_{i=1}^N$ generated by  OGDWOF throughout the run of AISPOY, return the points $\left({\bar{x} := \frac{1}{N}\sum_{i=1}^Nx_i, ~\bar{s} := \frac{1}{N}\sum_{i=1}^Ns_i}\right)$
\end{enumerate}
\end{algorithmic}
\end{algorithm}

\begin{lemma}\label{lemma:OCO-SP-xy}
The output of Algorithm \ref{OCO-SP-xy}, the tuple  $(\bar{x}, \bar{y})\in\reals^n\times\mX$, satisfies the following guarantees:
\begin{align*}
&i.~\max_{y \in  \mY} g(\bar{x},y) \leq \min_{x\in \alpha_{\mX}\mX} \max_{y \in \alpha_{\mY}^{-1}(\mY-\mB_+(R_{\mY}))} g(x,y)+ \frac{\text{Regret}_{OGDWOF}(g_1,\dots,g_N)}{\alpha_{\mY} N},   \\
&ii.~ \bar{s}\leq \bar{x}, \qquad iii.~ \Vert{\bar{x}}\Vert \leq (\alpha_{\mX}+2)R_{\mX}.
\end{align*}

In particular, if OGDWOF is applied with $\mu =\frac{\sqrt{2}\alpha_{\mX} R_{\mX}}{G_{\mX}\sqrt{N}}$ and $\xi =\frac{\alpha_{\mX} R_{\mX}G_{\mX}}{2\sqrt{N}}$, where $G_{\mX}$ is as defined in Lemma \ref{lemma:OCO-SP-y }, and with $A = A_{\ell,w} := \sum_{i=1}^dw(i)L_i$ for some $w\in\mB$ and the matrices $L_1,\dots,L_d$ defined in Lemma \ref{bi-linear lemma}, then 
\begin{align}\label{eq:OCO-SP-xy:1}
\max_{y \in \mY} \langle{w,\ell(\bar{x},y)}\rangle \leq \min_{x\in \alpha_{\mX} \mX} \max_{y \in \alpha_{\mY}^{-1}(\mY-\mB_+(R_{\mY}))}\langle{w,\ell(x,y)}\rangle+ \frac{(\sqrt{2}+1)\alpha_{\mX} R_{\mX} G_{\mX}}{\alpha_{\mY} \sqrt{N}},
\end{align}
and the algorithm requires $O \left( \frac{{(\alpha_{\mX} +2)}^4 {R_{\mX}}^2}{\alpha_{\mX}^2  G_{\mX}^2 } N^2\right)$ calls to $\oraclex$ and $N$ calls to $\oracley$.
\end{lemma}
\begin{proof}
The proof of inequality $i.$ follows exactly the same lines as in the proof of Lemma \ref{lemma:OCO-SP-y } with the obvious change that the regret guarantee of OGDWOF is w.r.t. the set $\alpha_{\mX}\mX$ as opposed to the set $\mX$ in the proof  of Lemma \ref{lemma:OCO-SP-y }. Note also that inequalities $ii.$ and $iii.$ follow in a straightforward manner  from the guarantees of Lemma \ref{lemma: infesProj}.

The second part of the Lemma, Eq. \eqref{eq:OCO-SP-xy:1}, follows directly from applying Lemma \ref{bi-linear lemma} together with the regret guarantee of OGDWOF given in Lemma \ref{lemma:infeasibile OGD}.  In particular, note that using  Lemma \ref{bi-linear lemma} we have that, for each iteration $i\in[N]$ of Algorithm \ref{OCO-SP-xy}, the cost vector $c_i$ sent to Algorithm OGDWOF, satisfies: $c_i = A_{\ell,w}^{\top}x_i$, and so, $\Vert{c_i}\Vert = \Vert{A_{\ell,w}^{\top}x_i}\Vert \leq \Vert{\ell}\Vert\Vert{x_i}\Vert \leq \Vert{\ell}\Vert(\alpha_{\mX}+2)R_{\mX}$, and thus, plugging-in our choices to parameters $\mu,\xi$ into Lemma \ref{lemma:infeasibile OGD}, we obtain that indeed $\text{Regret}_{OGDWOF}(g_1,\dots,g_N) \leq (\sqrt{2}+1)\alpha_{\mX}R_{\mX}G_{\mX}\sqrt{N}$. 

Finally, to upper-bound the number of calls to the approximation oracles, note that as in Algorithm \ref{OCO-SP-y}, the number of calls to $\oracley$ is simply $N$. As for the number of calls to $\oraclex$, Algorithm \ref{OCO-SP-y} makes $N$ calls to Algorithm \ref{alg:infeasible-proj} (through Algorithm \ref{alg:infeasibile OGD}), which according to Lemma \ref{lemma: infesProj}, amounts to overall $O\left({N\frac{(\alpha_{\mX}+2)^4R_{\mX}^4}{\xi^2}}\right)$ calls $\oraclex$. The bound now follows from plugging-in the value of $\xi$ stated in the lemma.

\end{proof}

\subsection{Putting everything together}
We are now ready to state in full detail and prove the third part of Theorem \ref{thm:main}.
\begin{theorem}\label{thoerem: app on xy} 
Let $(\mX,\mY,\ell,S)$ be a Blackwell instance that satisfies Assumption \ref{ass:1}. Consider running Algorithm \ref{oco-based app  alg} w.r.t. $(\mX,\mY,\ell,S)$ and the modified Blackwell instance:
$(\tilde{\mX}, \tilde{\mY}, \ell, \tilde{S}) = \left({\alpha_{\mX}\mX, \frac{1}{\alpha_{\mY}}(\mY-\mB_+(R_{\mY})), \ell, \frac{\alpha_{\mX}}{\alpha_{\mY}}(S-\mB_{+}(\tilde{R}))}\right)$,
where  $\tilde{R}:=\Vert \ell \Vert R_\mX R_\mY$, with the following implementation:
\begin{itemize}
\item
OGD with step-size $\eta= \frac{2}{G_5 T}$ is the OCO algorithm for the ball $\mB$.
\item
On each iteration $t\in[T]$, the computation $(x_t,s_t)\gets \oraclesp(w_t, \epsilon)$ is implemented using Algorithm \ref{OCO-SP-xy} as follows:
\begin{align*}
\left(x_t , s_t \right) \gets \textsc{AISPOYX}  \left( g_t(x,y) :=x^\top \left( \sum_{i=1}^d {w_t}(i)L_i \right) y, \oraclex, \oracley, N(T) , \mu ,\xi  \right),
\end{align*}
where $N(T):=\frac{(\sqrt{2}+1)^2\alpha_{\mX}^2 G_{\mX}^2 R_{\mX}^2 T}{{\alpha_{\mY}}^2}$, $\mu= \frac{\sqrt{2}\alpha_{\mX} R_{\mX}}{G_{\mX} \sqrt{N(T)}}$, $\xi= \frac{\alpha_{\mX} R_{\mX}G_{\mX}}{ 2\sqrt{N(T)}}$ .
\end{itemize}
Then, 
\begin{align}\label{eq:s_t conv x y}
d\left(\bar{\ell}_{s,T} , \left(\alpha_{\mX}\alpha_{\mY}^{-1} S \right)_\downarrow \right) \leq \frac{3G_5 +3}{\sqrt{T}} ,
\end{align}
and the overall number of calls to the approximation oracles of $\mX$ and $\mY$   is $O \left( \frac{{(\alpha_{\mX} +2)}^4 {R_{\mX}}^6 G_{\mX}^2 \alpha_{\mX}^2  }{\alpha_{\mY}^4}T^3\right)$ and $O \left( \frac{{\alpha_{\mX}}^2{G_\mX}^2 {R_{\mX}}^2 }{\alpha_{\mY}^2} T^2 \right)$, respectively, where $G_{\mX}$ is as defined in Lemma \ref{lemma:OCO-SP-y } and $G_5$ satisfies: 
\begin{align*}
G_5=\max_{s\in S,u \in \mB_+(\Vert \ell \Vert R_\mX R_\mY),x \in \mB((\alpha_{\mX} +2) R_{\mX}) ,y \in \mY} \left\Vert\frac{\alpha_{\mX}}{\alpha_{\mY}}(s- u) - \ell (x,y)\right\Vert.
\end{align*}
\end{theorem}
\begin{proof}
First, we note that the Blackwell instance $(\tX,\tY,\ell,\tS)$ is indeed approachable due to combination of Lemmas  \ref{lemma: alphaS app} and \ref{lemma:y_cond} (see appendix).

From Lemma \ref{lemma:OCO-SP-xy} we have that, with the implementation detailed in the theorem, Algorithm \ref{OCO-SP-xy} in indeed an infeasible saddle-point oracle w.r.t. the Blackwell instances $(\mX,\mY,\ell,S)$, $(\tX, \tY, \ell, \tS)$ with parameter $R_{SP} = (\alpha_{\mX}+2)R_{\mX}$ and with approximation error  $\epsilon = \frac{1}{\sqrt{T}}$. Thus, by invoking Theorem \ref{theorem:oco-based app  alg} using the OGD implementation detailed in the theorem, we have that
\begin{align*}
d\left(\bar{\ell}_{x,T} , \tS\right)\leq \frac{3 G_5}{\sqrt{T}} + \frac{1}{\sqrt{T}} = \frac{3G_5+1}{\sqrt{T}}.
\end{align*} 
Now, define $\ell_{T} := \argmin_{s \in \tS} \|\bar{\ell}_{x,T} - s\|$, $u_{T}:=\bar{\ell}_{x,T} - \ell_{T}$, and $\tilde{\ell}_{T}:=\bar{\ell}_{s,T} - u_{T}$. First, note that $\tilde{\ell}_T = \bar{\ell}_{s,T} - \bar{\ell}_{x,T} + \ell_T \leq \ell_T$ (since $\bar{\ell}_{s,T} \leq \bar{\ell}_{x,T}$), and since $\ell_T\in\tS = \alpha_{\mX}\alpha_{\mY}^{-1}(S-\mB_{+}(\tilde{R}))\subseteq\left({\alpha_{\mX}\alpha_{\mY}^{-1}S}\right)_{\downarrow}$, we have that $\tilde{\ell}_T\in(\alpha_{\mX}\alpha_{\mY}^{-1}S)_{\downarrow}$. Thus, we have that
\begin{align*}
d\left(\bar{\ell}_{s,T} ,{\left( \alpha_{\mX}\alpha_{\mY}^{-1} S\right)}_\downarrow \right) \leq \|\bar{\ell}_{s,T}-\tilde{\ell}_{T}\| =\|u_{T}\| = d\left(\bar{\ell}_{x,T} , \tS\right),
\end{align*}
and the first part of the theorem holds.

The number of calls to the approximation oracles is simply $T$ times (for the $T$ iterations of Algorithm \ref{oco-based app  alg}) the bounds given in Lemma \ref{lemma:OCO-SP-xy} when plugging-in our choice of $N(T)$.
\end{proof}

We note that we could have switched the type of updates in Algorithm \ref{OCO-SP-xy} so that the $y$ variable is updated via OGDWOF, and the $x$ variable is updated via a simple call to the extended approximation oracle (similarly to Algorithm \ref{OCO-SP_x}). This would have led to overall $O(T^2)$ (omitting all dependencies except for $T$) calls to $\oraclex$ and $O(T^3)$ calls to $\oracley$.

\section*{Acknowledgements}
Dan Garber is supported by the ISRAEL SCIENCE FOUNDATION (grant No. 2267/22).

\bibliographystyle{plain}
\bibliography{References}

\appendix

\section{Proof of Lemma \ref{bi-linear lemma}}
We first restate the lemma and then prove it.
\begin{lemma}[structure of bi-linear function]
Given a bi-linear function $\ell:\reals^n\times\reals^m\rightarrow\reals^d$, there exists matrices $L_1,\dots,L_d\in\reals^{n\times m}$ such that
\begin{align}\label{bi-linear form}
\forall x,y:\quad [\ell (x,y)]_i=x^\top L_i y, \quad i \in [d].
\end{align}
Consequently, for any $w\in\mB$, the matrix $A_{\ell,w} := \sum_{i=1}^dw(i)L_i$ satisfies
\begin{align*}
\langle{w,\ell(x,y)}\rangle = x^{\top}A_{\ell,w}y, \quad \textrm{and} \quad \Vert{A_{\ell,w}}\Vert \leq \Vert{\ell}\Vert,
\end{align*}
where we define $\Vert{\ell}\Vert := \max_{u\in\mB}\Vert{\sum_{i=1}^du(i)L_i}\Vert$.
\end{lemma}

\begin{proof}
Let $x \in \mX , y \in \mY , i \in [d]$. Since $\ell$ is bi-linear we have that,
\begin{align*}
[\ell (x,y)]_i &=\left[\ell \left(\sum_{j=1}^n x_j \e_j,\sum_{k=1}^m y_k \e_k\right)\right]_i \\
&=\left[\sum_{j=1}^n \sum_{k=1}^m x_j y_k \ell (\e_j,\e_k)\right]_i \\
&= \sum_{j=1}^n \sum_{k=1}^m x_j y_k \left[\ell (\e_j,\e_k)\right]_i,
\end{align*}
Thus, defining the matrices $\{L_i\}_{i=1}^d$ as:
\begin{align*}
L_i[j,k] := [\ell (\e_j,\e_k)]_i \quad(i,j,k)\in[d]\times[n]\times[m],
\end{align*}
we indeed have that for all $x,y$ and $i$, $[\ell(x,y)]_i = x^{\top}L_iy$, and the first part of the lemma holds.

The second part of the lemma follows from straightforward linear algebra.
\end{proof}


\section{Approachability  of Modified Blackwell Instances}
\begin{lemma}\label{lemma: alphaS app}
Let $(\mX,\mY ,\ell ,S)$ be an approachable Blackwell instance.
Then, for any $\alpha_{\mX} \geq 1$ and  $0 < \alpha_{\mY}  \leq 1$, $(\alpha_{\mX}\mX,\alpha_{\mY}^{-1}\mY ,\ell ,\alpha_{\mX}\alpha_{\mY}^{-1}S)$ is an approachable Blackwell instance.
\end{lemma}

\begin{proof}\label{proof: alphaS app}
Fix $w\in\mB$. Since $S$ is approachable, according to Theorem \ref{thm:approach} there exists $x_w \in \mathcal{X}$ such that
\begin{align}\label{eq:alphaS:1}
\langle w, \ell (x_w, y) \rangle - h_S(w) \leq 0 \quad \forall y \in \mathcal{Y}.
\end{align}
Denote $\tilde{x} = \alpha_{\mX} x_w \in \alpha_{\mX}  \mX$. Fix some $\tilde{y}\in\alpha_{\mY}^{-1}\mY$ and let $y\in\mY$ be such that $\tilde{y} = \alpha_{\mY}^{-1}y$.
Using the definition of the support function (Eq. \eqref{support function}) and the bi-linearity of $\ell$, we observe that
 \begin{align*}
\langle w, \ell (\tilde{x}, \tilde{y}) \rangle - h_{\alpha_{\mX}\alpha_{\mY}^{-1} S}(w) &= \alpha_{\mX} \alpha_{\mY}^{-1} \langle w, \ell (x_w, y) \rangle - \alpha_{\mX} \alpha_{\mY}^{-1}  h_S(w) \\
&= \alpha_{\mX} \alpha_{\mY}^{-1} \left(\langle w, \ell (x_w, y) \rangle - h_S(w) \right)\leq 0,
\end{align*}
where the last inequality follows Eq. \eqref{eq:alphaS:1}.

Thus, from Theorem \ref{thm:approach} if follows that $(\alpha_{\mX}\mX,\alpha_{\mY}^{-1}\mY ,\ell ,\alpha_{\mX}\alpha_{\mY}^{-1}S)$ is indeed an approachable Blackwell instance.

\end{proof}

\begin{lemma}\label{lemma:y_cond}
Let $(\mX ,\mY ,\ell ,S)$ be a Blackwell instance which satisfies Assumption \ref{ass:1}. Then, for any $R>0$, the Blackwell instance  $(\mX ,\tY ,\ell ,\tS)$, where $\tY := \mY- \mathcal{B}_+(R),\tS:= S-\mathcal{B}_+
(\tilde{R})$, for $\tilde{R}=\Vert \ell \Vert R_{\mX} R$ (where $\Vert \ell \Vert $ is defined in Lemma \ref{bi-linear lemma}), is  approachable.
\end{lemma}

\begin{proof}

Fix some $w\in\mB$. By the definition of the support function (Eq, \ref{support function}) we have that, 
\begin{align}\label{eq:alphaY:1}
h_{\tS}(w)&= h_{S-\mathcal{B}_+(\tilde{R})}(w)= h_{S}(w)+h_{-\mathcal{B}_+(\tilde{R})}(w) \notag \\ 
&= h_{S}(w)+\begin{cases}
    0 & \text{if } w >0; \\
   \tilde{R}\|w^-\| & \text{otherwise },
\end{cases}
\end{align}
where $w^-$ equals to $w$ on all negative coordinates of $w$ and zero everywhere else. 

Also, since $S$ is approachable, according to Theorem \ref{thm:approach} there exists $x_w \in \mathcal{X}$ such that
\begin{align}\label{eq:alphaY:1}
\langle w, \ell (x_w, y) \rangle - h_S(w) \leq 0 \quad \forall y \in \mathcal{Y}.
\end{align}
Denote $\tilde{x} = \alpha_{\mX} x_w \in \alpha_{\mX}  \mX$. Fix some $\tilde{y}\in\alpha_{\mY}^{-1}\mY$ and let $y\in\mY$ be such that $\tilde{y} = \alpha_{\mY}^{-1}y$.

Let $\tilde{y}\in\tilde{\mY}$ and let $y = \tilde{y}-b$ such that $b\in\mB_+(\tilde{R})$ and $y\in\mY$. Using Eq. \eqref{eq:alphaY:1} and the bi-linearity of $\ell$ we have that,
\begin{align}\label{eq:alphaY:2}
\langle w, \ell (x_{w}, \tilde{y}) \rangle - h_{\tS}(w) = \langle w, \ell (x, y) \rangle - \langle w, \ell (x_w, b)\rangle - h_{S}(w)-\tilde{R}\|w^-\|.
\end{align}
Since $x_w\in\mX\in\reals^n_+$ (under Assumption \ref{ass:1}) and $b\in\mB_+(R)$, we have under Assumption \ref{ass:1} that $\ell(x_w,b) \geq 0$. Thus, we have that 
\begin{align*}
 \langle w, \ell (x_w, b)\rangle &\geq  \langle w^-, \ell (x_w, b)\rangle \underset{(1)}{\geq} -\Vert{w^-}\Vert{}\langle \vec{w}^-, \ell (x_w, b)\rangle \\
 &\underset{(2)}{\geq} 
 -\Vert{w^-}\Vert{}\Vert{\ell}\Vert\Vert{x_w}\Vert\Vert{b}\Vert \underset{(3)}{\geq} -\Vert{w^-}\Vert\Vert{\ell}\Vert{}R_{\mX}R,
\end{align*}
where in (1) we let $\vec{w}^- = w^-/\Vert{w^-}\Vert$ if $w^-\neq 0$ and $\vec{w}^-$ is some arbitrary unit vector otherwise, in (2) we use Lemma \ref{bi-linear lemma}, and in (3) we plug-in the bounds on $b$ and $x_w\in\mX$. 

Thus, plugging-in the above inequality into \eqref{eq:alphaY:2} and using the fact that $\tilde{R} = \Vert{\ell}\Vert{}R_{\mX}R$, we have that
\begin{align*}
\langle w, \ell (x_{w}, \tilde{y}) \rangle - h_{\tS}(w) &\leq  \langle w, \ell (x, y) \rangle + \Vert{w^-}\Vert\Vert{\ell}\Vert{}R_{\mX}R- h_{S}(w)-\tilde{R}\|w^-\| \\
&=  \ell (x, y) \rangle - h_{S}(w) \leq 0,
\end{align*}
where the last inequality is due to Eq. \eqref{eq:alphaY:1}.

Thus, we can conclude that $(\mX ,\tY ,\ell ,\tS)$ is indeed an approachable Blackwell instance.

\end{proof}

\section{Proof of Lemma \ref{lemma:infeasibile OGD}}
We first restate the lemma and then prove it.
\begin{lemma}
Suppose Assumption \ref{ass:1} holds. The sequences $(x_t)_{t=1}^N\subset\mB((\alpha_{\mX}+2)R_{\mX})$, $(s_t)_{t=1}^N\subset\mX$ produced by Algorithm \ref{alg:infeasibile OGD} satisfy the following two guarantees:
\begin{align*}
&i.~ \forall x \in \alpha_\mX \mX: \quad \sum_{t=1}^N \langle x_t,c_t \rangle -\sum_{t=1}^N \langle x,c_t \rangle \leq  \frac{\alpha_{\mX}^2R_{\mX}^2}{\mu }  +\frac{\mu }{2} \sum_{t=1}^N \|c_t\|^2 + 2\xi N, \\
&ii.~  \frac{1}{N}\sum_{t=1}^Ns_t \leq \frac{1}{N}\sum_{t=1}^N x_t.
\end{align*}
\end{lemma}
\begin{proof}
Fix some $x \in \alpha_\mX \mX$. Using Lemma \ref{lemma: infesProj}, for any iteration $t \geq 1$ it holds that,
\begin{align*}
\|x_{t+1} -x \|^2 &\leq  \|y_{t+1} -x \|^2 +2\xi = \|x_t -\mu c_t -x \|^2 +2\xi \\
&=  \|x_t -x \|^2 + \mu^2 \|c_t \|^2 - 2\mu\langle{x_t -x, c_t}\rangle +2\xi.
\end{align*}
Rearranging and summing over all iterations we get,
\begin{align*}
\sum_{t=1}^N \langle{x_t - x, c_t}\rangle &\leq \frac{1}{2 \mu} \sum_{t=1}^N\left( \|x_t-x\|^2 - \|x_{t+1}-x\|^2 \right) + \frac{\mu}{2} \sum_{t=1}^N \|c_t\|^2 + 2\xi{}N \\
&\leq \frac{1}{2 \mu } \|x_1-x\|^2 +\frac{\mu }{2} \sum_{t=1}^N \|c_t\|^2 + 2\xi{}N\\
&\leq \frac{\alpha_{\mX}^2 R_{\mX}^2}{ \mu } +\frac{\mu }{2} \sum_{t=1}^N \|c_t\|^2 + 2\xi  N,
\end{align*}
where the last inequality follows since both $x,x_1$ are in $\alpha_{\mX}\mX\subset\reals^n_+$ (under Assumption \ref{ass:1}).

The inequality $\frac{1}{N}\sum_{t=1}^N s_t \leq \frac{1}{N}\sum_{t=1}^N x_t$ follows immediately from Lemma \ref{lemma: infesProj}, which assures that for any $t\in[N]$, it holds that  $s_t \leq x_t$.
\end{proof}




\end{document}